\documentclass[oneside,english]{amsart}
\usepackage[T1]{fontenc}
\usepackage[latin9]{inputenc}
\usepackage{amstext}
\usepackage{amsthm}
\usepackage{amssymb,enumerate}
\usepackage[colorlinks=true]{hyperref}
\hypersetup{
	colorlinks=blue,%
	citecolor=red,%
	filecolor=black,%
	linkcolor=blue,%
	urlcolor=blue
}

\makeatletter
\theoremstyle{plain}
\newtheorem{thm}{\protect\theoremname}
\theoremstyle{plain}
\newtheorem{lem}[thm]{\protect\lemmaname}
\theoremstyle{plain}
\newtheorem{prop}[thm]{\protect\propositionname}
\theoremstyle{plain}

\makeatletter

\usepackage{babel}
\providecommand{\lemmaname}{Lemma}
\providecommand{\theoremname}{Theorem}
\providecommand{\propositionname}{Proposition}

\def \pa {\partial}

  \newcommand{\R}{\mathbb{R}}
  
 \newcommand{\ud}{\,\mathrm{d}}
\newcommand{\Sn}{\mathbb{S}^n}
\newcommand{\Sp}{\mathbb{S}}

\allowdisplaybreaks
\begin{document}
\title[]{Obata theorem of $\sigma_k$-flat and constant  associated $k$th   boundary curvature equation}

\author{Xuezhang  Chen, Wei Wei}
\address{School of Mathematics \& IMS, Nanjing University, Nanjing 210093, P.R. China}
\email{xuezhangchen@nju.edu.cn,wei\_wei@nju.edu.cn}

\begin{abstract}
On the upper hemisphere, we use the Obata-Escobar argument to establish an Obata theorem for 
$\sigma_k$-flat and  positive constant boundary $\mathcal B_k$-curvature
equation, which was initiated by Case-Wang [Adv.  Math. 337 (2018), 83-106]. 

\medskip

\end{abstract}
\maketitle

\section{Introduction}

Viaclovsky
\cite{Viaclovsky2000a,Viaclovsky2000b} initiated the systematic study of the
conformal properties of the $\sigma_k$-curvatures. Since then, substantial
progress has been made on the $\sigma_k$-Yamabe problem and related fully
nonlinear equations. We refer to
Chang-Gursky-Yang \cite{ChangGurskyYang2002a,ChangGurskyYang2002b,
ChangGurskyYang2003}, Guan-Wang \cite{GuanWang2003,GuanWang2004}, Li-Li
\cite{LiLi2003,LiLi2005}, Sheng-Trudinger-Wang
\cite{ShengTrudingerWang2007} and references therein. 

To continue, we first introduce some notations. 
 For an $N \times N$ matrix $B$ and $k \in \mathbb{Z}_+$, as in \cite{Chen2009} we define by
\[\sigma_k(B)=\frac{1}{k!} \delta_{i_1 \cdots i_k}^{j_1 \cdots j_k} B_{j_1}^{i_1} \cdots B_{j_k}^{i_k}
\]
the $k$-th elementary symmetric function and we introduce a mixed symmetric function by
\[\sigma_{k, \ell}(B, C)=\frac{1}{k!} \delta_{i_1 \cdots i_k}^{j_1 \cdots j_k} B_{j_1}^{i_1} \cdots B_{j_{\ell}}^{i_{\ell}} C_{j_{\ell+1}}^{i_{\ell+1}} \cdots C_{j_k}^{i_k}\]
for $N \times N$ matrices $B,C$, and assign $\sigma_0=1$.
Define 
$$
\Gamma_k^{+}:=\left\{B \mid \sigma_1(B), \cdots, \sigma_k(B)>0\right\}.
$$

Let $N=n+1\geq 3$. For a smooth compact Riemannian manifold $(X^N,g)$ with boundary $M^n=\pa X$, let $\nabla$ and $\overline{\nabla}$ be the Levi-Civita connections on $X$ and $M$, respectively. 

Denote by
\[
A_{g}:=\frac{1}{N-2}\Big(\mathrm{Ric}_{g}-\frac{R_{g}}{2(N-1)}g\Big)
\]
the Schouten tensor, where $\mathrm{Ric}_{g}$ and $R_{g}$ are the Ricci tensor and the scalar
curvature in $g$, respectively. Define 
\[\mathcal{C}_{k}^{+}=\{g\mid  g^{-1}A_g\in \Gamma_k^+\}.\]

The variational $k$th boundary  curvature associated to
$\sigma_k(A_g)$ was introduced by S. Chen
\cite[(2)-(3) on p.1030]{Chen2009}: On the boundary $M$ and for $N>2k$,
 $$
\mathcal{B}_{k}^{g}=\sum_{j=0}^{k-1} \frac{(2 k-j-1)!(N-2 k+j)!}{(N-k)!(2 k-2 j-1)!!j!} \sigma_{2 k-j-1, j}\left(A_{g}^{\top}, L_g\right),
$$
where, $A_{g}^{\top}$ denote the tangential part of $A_{g}$, and $L_g:=\nabla \nu_g$
is the second fundamental form with respect to $g$, with its tracefree part/umbilic tensor
$$\mathring L_g=L_g-H_g \bar g$$
being a conformal invariant.   Here, $\nu_{g}$ is the outward unit 
normal vector and $\bar g=g\big|_{\pa M}$.  The boundary $\mathcal B_2$-curvature also naturally appears in the Chern-Gauss-Bonnet formula.

For $k \geq 2$, assume in addition that $g$ is  locally conformally flat and the boundary is umbilic, it has the form of (cf. Case-Wang \cite[Lemma 5.3]{CaseWang2018}) 
\[\mathcal{B}_{k}^{g}=\sum_{j=0}^{k-1} \frac{(n-j)!}{(n+1-k)!(2 k-2 j-1)!!} H_g^{2 k-2 j-1} \sigma_{j}(A_{g}^{\top}),\]
where
$H_{g}=\sum_{\alpha=1}^{n}L_{~\alpha}^{\alpha}/n$ is the mean curvature. The definition of $\mathcal{B}_{k}^{g}$ is also adapted to $k=1$. Especially, the pair $(\sigma_1, \mathcal{B}_{1}^{g})$ coincides with $(R_g,  H_g)$ up to some constants. The classification of the case of scalar curvature with boundary has been a great success traced back to the classical work of Escobar \cite{Escobar1990}.  For $(\sigma_k, H_g)$, it was also studied  in \cite{LiLi2006, ChenWei2026}. However, for $k\ge 2$, the boundary equation is a quasilinear or fully nonlinear type equation, which is rare in the literature.

Recently, the Liouville theorems with the boundary curvature $\mathcal B_k$ have  been
developed for positive $\sigma_k$ curvature. Wei \cite{Wei2026} found that the boundary equation is also elliptic, and 
classified solutions in the half-space under an assumption on the behavior of
the Kelvin transform at infinity. Chu-Li-Li \cite{ChuLiLi2023}
subsequently removed this asymptotic assumption.

For the degenerate case that $\big(\sigma(A_g),\mathcal{B}_k^g\big)=(0,c_0)$ with $c_0 \in \R_+$,  Case-Wang \cite{CaseWang2018}  established a first rigidity theorem  in $\mathbb S_+^{n+1}$: An admissible conformal metric is flat assuming  the additional pinching condition
\[
 \sup_{\mathbb S^n}H_g\leq (k+1)\inf_{\mathbb S^n}H_g.
\]
They further indicated that this pinching assumption might be redundant. See \cite{CaseWang-JMS} for some special cases. 
Case-Moreira-Wang \cite{CaseMoreiraWang2019}
obtained bifurcation and multiplicity results in a different geometric
setting. The Obata theorem and its proof are of interest in themselves.

Our result gives a positive complete answer to the Obata theorem initiated by Case-Wang \cite[Theorem 1.2]{CaseMoreiraWang2019}.
\begin{thm}
\label{thm:Main}
Let $g$ be a conformal metric in $(\mathbb{S}_{+}^{n+1},g_{\mathbb{S}^{n+1}})$ such that $g \in \overline{\mathcal C_k^+}$ for $k\geq 2$. Assume that $g$ is a smooth solution to
\begin{equation}\label{eqs:curvature}
\begin{cases}
\sigma_{k}(A_g)=0 & \quad \mathrm{in}\quad\mathbb{S}_{+}^{n+1},\\
\mathcal{B}_{k}^{g}=c_{0} \in \R_+ & \quad  \mathrm{on}\quad \pa \mathbb{S}_{+}^{n+1}.
\end{cases}
\end{equation}
Then $g$ is flat and $g\big|_{\pa \mathbb{S}_{+}^{n+1}}$ is a round metric.
\end{thm}

Write $g=u^2 g_{\mathbb{S}_+^{n+1}}$. Using Obata-Escoba argument as in Case-Wang \cite{CaseWang2018}, we obtain
\[-(k+1)\int_{\mathbb{S}_{+}^{n+1}}\sigma_{k+1}\left(A_{g}\right)u \ud v_{g}\le \int_{\mathbb{S}^{n}}H_{g}uT_{k}(A)(\nu_{g},\nu_{g})\ud\sigma_{g}.\]

An auxiliary function $\psi_{f}$ for $f=ux_{n+2}$ enables us to transfer the following \emph{boundary integral to an interior integral} 
\begin{align*}
\int_{\mathbb{S}^{n}}H_{g}uT_{k}(A)(\nu_{g},\nu_{g})\ud\sigma_{g}&=\int_{\mathbb{S}^{n}}\psi_{f}uT_{k}(A_g)(\nu_{g},\nu_{g})\ud\sigma_{g} \\
& =  \int_{\mathbb{S}_{+}^{n+1}}\mathrm{div}(T_{k}(A_g)\psi_{f}X)\ud v_{g},
 \end{align*}
where $\psi_{f}=-H_g$ and $X=-u\nu_g$ on $\pa \mathbb S_+^{n+1}$; see Proposition \ref{thm:various properties of uxn}.

By exploring `\emph{nice}' properties of $\psi_f$, we finally arrive at 
\[
\int_{\mathbb{S}_{+}^{n+1}}\big[-u\sigma_{k+1}(A_{g})\psi_{f}f+\frac{1}{2}u\sigma_{k+1}(A_{g})x_{n+2}^{2}\big]\ud v_{g}=0.
\]
Then the rigidity result follows.

\bigskip

During our preparation, we became aware of the preprint of Wangzhe Wu \cite{Wu2026Rigidity}, who dealt with the case $k=2$. The present work was carried out independently of his, and the method is different from his.

\medskip
\noindent{\em Acknowledgments.}
Wei Wei would like to thank Professor Yi Wang for her interest and the nice conversations about the boundary $\mathcal B_k$-curvature.

\

\section{Preliminaries}

From now on, we let $g=u^{2}g_{0}$ in $\mathbb{S}_{+}^{n+1}:=\{x\in \Sp^{n+1}; x_{n+2}>0\}$ with boundary $\mathbb{S}^{n}:=\pa \mathbb{S}_{+}^{n+1}$, where $g_{0}=g_{\mathbb{S}^{n+1}}$ is the round metric.

Let 
$$
T_k(A_{g})_i^j=\frac{1}{k!} \delta_{i i_1 \cdots i_k}^{j j_1 \cdots j_k} A_{j_1}^{i_1} \cdots A_{j_k}^{i_k}
$$
be the $k$th Newton transformation. Since $g$ is locally conformally flat, $T_{k}(A_{g})$ is divergent-free; see \cite{Viaclovsky2000a}. 

Denote by
\[
L_{k}(A_{g}):=\frac{N-k}{N}\sigma_{k}(A_{g})g-T_{k}(A_{g})
\]
 the tracefree $k$th Newton transformation $T_k(A_g)$, and by
\begin{equation}
E_{g}=\mathrm{Ric}_{g}-\frac{R_{g}}{N}g=(n-1)A_{g}-\frac{n-1}{2(n+1)n}R_{g}g\label{eq:expression of E by schouten tensor}
\end{equation}
 the tracefree Ricci tensor.
 
 A. Li-Y. Y. Li \cite[Lemma 6.8]{LiLi2003} states that 
 \begin{equation}\label{eq:algebra}
 \langle L_{k}\left(A_{g}\right),E_{g}\rangle_g
 = (n-1)\left [ \frac{n+1-k}{n+1}\sigma_{k}\left(A_{g}\right)\sigma_{1}\left(A_{g}\right)-(k+1)\sigma_{k+1}\left(A_{g}\right)\right ].
 \end{equation}
 The traceless operator $L_{k}(A_{g})$ has played an important role in $\sigma_{k}$ Yamabe equations, see for example
\cite{ChangGurskyYang2003,ChangGurskyYang2002a,LiLi2005}, and references therein.  

By the conformal transformation rules of $E_g$ and $H_g$, we obtain
\begin{equation}\label{Einstein:conf-change}
E_{g}=-\frac{n-1}{u}(\nabla^{2}u-\frac{\Delta_{g}u}{n}g),
\end{equation} 
since $g_0$ is Einstein, and 
\begin{equation}\label{eq:mean-curv}
\frac{\pa u}{\pa \nu_{g}}=H_{g}u \quad \mathrm{~~on~~}  \mathbb{S}^{n}.
\end{equation}

For readers' convenience, we restate  \cite[Lemma 5.3]{CaseWang2018} here.
\begin{lem}
\label{lem:Case-Wang} 
For $k\in\mathbb{Z}_+$, let $\left(X^{n+1},g\right)$
be a compact Riemannian manifold with umbilic boundary $M$;
for $k\geq2$, assume in addition that $g$ is locally conformally
flat. Then
\begin{align*}
&\int_{M}T_{k}(A_g)(\nu_{g},\overline{\nabla} u) \ud\sigma_{g}+\frac{N-k}{N-1}\int_{M}\mathcal{B}_{k}^{g}\overline{\Delta}u \ud\sigma_{g}\\
=&-\sum_{j=0}^{k-1}\frac{(n-j-1)!}{(n-k)!(2k-2j-1)!!}\int_{M}H_{g}^{2k-2j-1}\langle -L_{j}^{*}(A_{g}^{\top}),\overline{\nabla}^{2}u\rangle \ud\sigma_{g}
\end{align*}
for all $u\in C^{\infty}(M)$.  Here,
\[
L_{j}^{*}(A_{g}^{\top})=\frac{n-j}{n}\sigma_{j}(A_{g}^{\top})g-T_{j}(A_{g}^{\top}).
\]
\end{lem}

In the following, suppose $g=u^2 g_0$ is a smooth solution to \eqref{eqs:curvature}. We first collect some common facts:
\begin{enumerate}[(a)]
\item The boundary $\pa \mathbb{S}_+^{n+1}$ is umbilic with respect to $g$, that is, $L_g=H_g \bar g$. By \eqref{eq:mean-curv} we have
\begin{equation}\label{bdry:Hess-u}
\nabla^{2}u=\overline{\nabla}^{2}u+L_g \frac{\pa u}{\pa \nu_g}= \overline{\nabla}^{2}u+H_g^2 u \bar g \qquad \mathrm{on~~} \Sn.
\end{equation}
\item By the assumption that $A_g \in \overline{\mathcal C_k^+}$ and $\sigma_k(A_g)=0$,  $T_{k-1}(A_g)$ is semi-positive definite and $\sigma_{k+1}(A_g) \leq 0$ by Maclaurin inequality.
\end{enumerate}
\begin{lem}[\protect{\cite[Lemma 5.4]{CaseWang2018}}]\label{Case-Wang KZ}
 For $k\geq 2$, let $g=u^2 g_0\in  \overline{\mathcal C_k^+}$ satisfy $\sigma_k(A_g)=0$ in $\left(\mathbb S_{+}^{n+1}, g_0\right)$. Then
\begin{equation}
\int_{\mathbb{S}^n} u \sigma_{k}(A_{g}^{\top}) \ud v_{g}=0 .
\end{equation}

\end{lem}

\begin{lem} \label{lem:Inequality} 
Let $g=u^2 g_0 \in \overline{\mathcal C_k^+}$ be a solution to \eqref{eqs:curvature}. Then
\begin{equation}\label{eq:final formula 1}
\begin{aligned}
&-(k+1)\int_{\mathbb{S}_{+}^{n+1}}\sigma_{k+1}\left(A_{g}\right)u \ud v_{g}\\
= & -\sum_{j=0}^{k-1}\frac{(n-j-1)!}{(n-k)!(2k-2j-1)!!}\int_{\mathbb{S}^{n}}H_{g}^{2k-2j-1}\langle L_{j}^{*}(A_{g}^{\top}),uA_{g}^{\top}\rangle \ud\sigma_{g}\\
&+\int_{\mathbb{S}^{n}}H_{g}uT_{k}(A)(\nu_{g},\nu_{g})\ud\sigma_{g}\\
\le& \int_{\mathbb{S}^{n}}H_{g}uT_{k}(A)(\nu_{g},\nu_{g})\ud\sigma_{g},
\end{aligned}
\end{equation}
with equality if and only if $A_{g}^{\top}=0$ on $\mathbb{S}^{n}$.
\end{lem}

\begin{proof}
By (\ref{eq:algebra}) and \eqref{Einstein:conf-change}, we have 
\begin{align*}
&-(n-1)(k+1)\int_{\mathbb{S}_{+}^{n+1}}\sigma_{k+1}\left(A_{g}\right)u \ud v_{g}\\
 = & \int_{\mathbb{S}_{+}^{n+1}}\langle L_{k}\left(A_{g}\right),E_{g}\rangle_g u\ud v_{g}\\
 = & \int_{\mathbb{S}_{+}^{n+1}}\langle L_{k}\left(A_{g}\right),-(n-1)\nabla^{2}u\rangle_{g}\ud v_{g}\\
 = & \int_{\mathbb{S}_{+}^{n+1}}\langle \frac{n+1-k}{n+1}\sigma_{k}(A_{g})g-T_{k}(A_{g}),-(n-1)\nabla^{2}u\rangle_{g}\ud v_{g}\\
 = & (n-1)\int_{\mathbb{S}_{+}^{n+1}}\langle T_{k}(A_{g}),\nabla^{2}u \rangle_g \ud v_{g}\\
 = & (n-1)\int_{\mathbb{S}^{n}}T_{k}(A_{g})(\nabla u,\nu_{g})\ud\sigma_{g}\\
 = & (n-1)\int_{\mathbb{S}^{n}}T_{k}(A_{g})(\overline{\nabla}u+\frac{\pa u}{\pa \nu_{g}}\nu_{g},\nu_{g})\ud\sigma_{g}\\
 = & (n-1)\left[\int_{\mathbb{S}^{n}}T_{k}(A_{g})(\overline{\nabla}u,\nu_{g})\ud\sigma_{g}+\int_{\mathbb{S}^{n}}\frac{\pa u}{\pa \nu_{g}}T_{k}(A)(\nu_{g},\nu_{g})\ud\sigma_{g}\right].
\end{align*}
Hence, by Lemma \ref{lem:Case-Wang} and \eqref{bdry:Hess-u} we obtain
\begin{align*}
&-(k+1)\int_{\mathbb{S}_{+}^{n+1}}\sigma_{k+1}\left(A_{g}\right)u \ud v_{g}\\
 =&\int_{\mathbb{S}^{n}}T_{k}(A_{g})(\overline{\nabla}u,\nu_{g})\ud\sigma_{g}+\int_{\mathbb{S}^{n}}H_{g}uT_{k}(A)(\nu_{g},\nu_{g})\ud\sigma_{g}\\
 =&-\sum_{j=0}^{k-1}\frac{(n-j-1)!}{(n-k)!(2k-2j-1)!!}\int_{\mathbb{S}^{n}}H_{g}^{2k-2j-1}\langle -L_{k}^{*}(A_{g}^{\top}),\overline{\nabla}^{2}u\rangle \ud\sigma_{g}\\
 & +\int_{\mathbb{S}^{n}}H_{g}uT_{k}(A)(\nu_{g},\nu_{g})\ud\sigma_{g}\\
=&-\sum_{j=0}^{k-1}\frac{(n-j-1)!}{(n-k)!(2k-2j-1)!!}\int_{\mathbb{S}^{n}}H_{g}^{2k-2j-1}\langle -L_{k}^{*}(A_{g}^{\top}),-uA_{g}^{\top}\rangle \ud\sigma_{g}\\
& +\int_{\mathbb{S}^{n}}H_{g}uT_{k}(A)(\nu_{g},\nu_{g})\ud\sigma_{g}.
\end{align*}
Since $\langle L_{j}^{*}(A_{g}^{\top}),uA_{g}^{\top}\rangle \geq 0$ for each $0\leq j \leq k-1$, we immediately obtain \eqref{eq:final formula 1}.

When the equality holds, it follows that $\langle L_{j}^{*}(A_{g}^{\top}),uA_{g}^{\top}\rangle=0$ for $0\le j\le k-1$. Especially, for $j=1$, we have \[\langle L_{1}^{*}(A_{g}^{\top}),A_{g}^{\top}\rangle=|L_{1}^{*}(A_{g}^{\top})|_g^2=0.\] That is, \(A_{g}^{\top}=\lambda g\) on $\mathbb S^n$ for some non-negative function $\lambda$. By Lemma \ref{Case-Wang KZ}, we know that $\lambda=0$, and we thus finish the proof. 
\end{proof}
For the conformal factor $f=ux_{n+2}$, Gover-Nurowski \cite{GoverNurowski2006} proved some related properties of $f$ in the Poincar\'e-Einstein manifold setting, for example (\ref{eq:hessf+fA=psig})
and (\ref{eq:property of uxn 2}) below.
\begin{prop}
\label{thm:various properties of uxn}
  Let $g_+=x_{n+2}^{-2}g_{0}$ be the Poincar\'e-Einstein metric in $\mathbb{S}_{+}^{n+1}$. Assume that a compactified metric $g=f^2 g_+ \in [g_0]$ satisfies $g \in \overline{\mathcal C_k^+}$ and $\sigma_k(A_g)=0$. Then there exists a smooth function $\psi_{f}$ such that
\begin{equation}\label{eq:hessf+fA=psig}
\nabla^2 f+fA_{g}=\psi_{f}g\quad\mathrm{~~in~~}\mathbb{S}_{+}^{n+1}; \quad
\psi_{f}=-H_{g} \qquad\mathrm{~~on~~}\mathbb{S}^{n}.
\end{equation}
Moreover, $\psi_f$ enjoys the following properties.
\begin{equation}\label{eq:property of uxn 2}
\nabla \psi_{f}=-A_{g}\nabla f\quad\mathrm{~~in~~}\mathbb{S}_{+}^{n+1};
\end{equation}
\begin{equation}
|\nabla f|_{g}^{2}=1+2f\psi_{f};\label{eq:property of uxn}
\end{equation}
\begin{equation}
\mathrm{div}(T_{k-1}(A_g)\nabla\psi_{f})=-(k+1)f\sigma_{k+1}(A_g)\quad\mathrm{~~in~~}\mathbb{S}_{+}^{n+1};\label{eq:property of uxn 5}
\end{equation}
\begin{equation}
-\sigma_{k+1}(A_{g})\psi_{f}\le0 \quad\mathrm{~~in~~} \mathbb{S}_{+}^{n+1}.\label{eq:nonnegativity of psif}
\end{equation}
\end{prop}

\begin{proof}
 Write  $g=f^{2}g_+=u^{2}g_{0}$ with $f=u x_{n+2}$. Clearly,  $f=0$ on $\mathbb{S}^{n}$. 
 
For the compactified metric  $g=f^2 g_+$, we apply \eqref{eq:expression of E by schouten tensor} and \eqref{Einstein:conf-change} to conclude  that there exists a smooth function 
 $\psi_{f}$ such that 
\begin{equation}\label{eq:hess equality}
\nabla^2 f+fA_{g}=\psi_{f}g\quad\mathrm{~~in~~}\mathbb{S}_{+}^{n+1}.
\end{equation}
On $\mathbb{S}^{n}$, $\nabla_{g_{0}}x_{n+2}=-u\nu_{g}$, thereby
\begin{equation}\label{value:nu(f)}
\frac{\pa f}{\pa \nu_{g}}=\nu_g(ux_{n+2})= u\nu_g( x_{n+2})=-|\nabla x_{n+2}|_{g_0}^2=-1.
\end{equation}

For brevity, we use $f_{ij}=\nabla_i \nabla_j f$, so is $f_{\alpha \beta}$.

On one hand, there holds
\begin{align*}
f_{\alpha\beta} = & \overline \nabla_\beta \overline \nabla_\alpha f+H_{g}g_{\alpha\beta}\frac{\pa f}{\pa \nu_{g}}
 = H_{g}g_{\alpha\beta}\frac{\pa f}{\pa \nu_{g}}=-H_{g}g_{\alpha\beta}.
\end{align*}
On the other hand, by (\ref{eq:hess equality}) we have 
\[
f_{\alpha\beta}=f_{\alpha\beta}+fA_{\alpha\beta}=\psi_{f}g_{\alpha\beta}.
\]
Combining these two together yields $\psi_{f}=-H_{g}$ on $\mathbb{S}^{n}$. 

Taking covariant derivative to \eqref{eq:hess equality} yields
\begin{align*}
f_j=&\nabla^{i}\psi_{f}g_{ij} =  \nabla^{i}f_{ij}+f\nabla^{i}A_{ij}+\nabla^{i}fA_{ij}\\
 = & \nabla_{j}f_{i}^{~i}+R_{j}^{i}f_{i}+f\sigma_{1}(A_{g})_{j}++\nabla^{i}fA_{ij}\\
 = & ((n+1)\psi_{f}-f\sigma_{1}(A_{g}))_{j}+((n-1)A_{j}^{i}+\sigma_{1}(A_{g})\delta_{j}^{i})f_{i}+f\sigma_{1}(A_{g})_{j}+\nabla^{i}fA_{ij}\\
 = & (n+1)(\psi_{f})_{j}+(n-1)A_{j}^{i}f_{i}++\nabla^{i}fA_{ij},
\end{align*}
where $A_{ij}=\frac{1}{n-1}(R_{ij}-\frac{1}{2n}R_{g}g_{ij})$.
This implies 
\[
\nabla\psi_{f}=-A_{g}\nabla f\quad\mathrm{~~in~~}\mathbb{S}_{+}^{n+1}.
\]

Using \eqref{eq:hess equality} and \eqref{eq:property of uxn 2} to show 
\[
\nabla_{k}|\nabla f|_g^{2}=2f_{p}f_{k}^{p}=2f_{p}(-fA_{k}^{p}+\psi_{f}\delta_{k}^{p})=2f\nabla_{k}\psi_{f}+2f_{k}\psi_{f}=2\nabla_{k}(\psi_{f}f),
\]
we obtain that $|\nabla f|_{g}^{2}-2\psi_{f}f$ is constant in $\mathbb{S}_{+}^{n+1}$.
Meanwhile,  by \eqref{value:nu(f)} we obtain
\[
|\nabla f|_{g}^{2}-2\psi_{f}f=|\nabla f|_g^2=(\frac{\pa f}{\pa \nu_g})^2=1 \qquad \mathrm{on~~} \mathbb{S}^{n}.
\]
Hence, we conclude that $|\nabla f|_{g}^{2}-2\psi_{f}f=1$ in $\mathbb{S}_{+}^{n+1}$.

  By (\ref{eq:property of uxn}) and \eqref{eq:property of uxn 2} we have 
\begin{align*}
\mathrm{div}(T_{k-1}(A_g)\nabla\psi_{f}) = & \mathrm{div}(T_{k-1}(A_g)(-A_{g}\nabla f))\\
 = & -\nabla^{j}((T_{k-1})_{j}^{i}A_{i}^{p}f_{p})=-\nabla^{j}(\sigma_{k}(A)\delta_{j}^{p}f_{p}-(T_{k})_{j}^{p}f_{p})\\
 = & T_{k}(A_g)_{j}^{p}f_{p}^{j}=T_{k}(A_g)_{j}^{p}(\psi_{f}\delta_{p}^{j}-fA_{p}^{j})\\
 = &(n-k) \sigma_k(A_g)\psi_f -(k+1)f\sigma_{k+1}(A_g)\\
 =&-(k+1)f\sigma_{k+1}(A_g).
\end{align*}

Let $\psi_{f,+}=\max\{\psi_{f},0\}$. Since $\psi_f=-H_g<0$, the support of $\psi_{f,+}$ is compact in $\mathbb{S}_{+}^{n+1}$. Then multiplying the above equation by $\psi_{f,+}$ and integrating it over $\mathbb{S}_{+}^{n+1}$ give
\begin{align*}
 0\leq -(k+1)\int_{\mathbb{S}_{+}^{n+1}}f\sigma_{k+1}(A_g)\psi_{f,+} \ud v_{g}= &\int_{\mathbb{S}_{+}^{n+1}}\psi_{f,+}\mathrm{div}(T_{k-1}(A_g)\nabla\psi_{f}) \ud v_{g}\\
 = & -\int_{\{\psi_{f}>0\}}T_{k-1}(A_g)(\nabla\psi_{f},\nabla\psi_{f})\ud v_{g}\le0.
\end{align*}
This forces 
\[
0=\int_{\mathbb{S}_{+}^{n+1}}f\sigma_{k+1}(A_g)\psi_{f,+}\ud v_{g}=\int_{\{\psi_{f}>0\}}f\sigma_{k+1}(A_g)\psi_{f}\ud v_{g}.
\]
Thus we obtain $\sigma_{k+1}(A_g)=0$ in $\{\psi_{f}>0\}$. Again using $\sigma_{k+1}(A_g)\le0$, we conclude that $\sigma_{k+1}(A_g)\psi_{f}\ge0$ in $\mathbb{S}_{+}^{n+1}$.
\end{proof}

\section{Proof of Theorem \ref{thm:Main}}

\begin{proof}
[Proof of Theorem  \ref{thm:Main}] 
Write $g=u^2 g_0=f^2 g_+$, where $g_+=x_{n+2}^{-2}g_{0}$ and $f=u x_{n+2}$.
Let $X=-u\nu_{g}=\nabla_{g_0} x_{n+2}$ be the conformal Killing vector field on $(\mathbb{S}_{+}^{n+1},g_0)$. Denote by
$$(\mathfrak{D}_g X)_{ij}:=\nabla_i X_j+\nabla_j X_i-\frac{2}{n+1} \mathrm{div}_g(X)g_{ij}$$
the conformal Killing operator with the property that
$$\mathfrak{D}_g X=u^2 \mathfrak{D}_{g_0} X=0.$$
Then we obtain
\begin{equation}\label{key term:vanishing}
T_{k}(A_g)_{j}^{i}\nabla_{i}X^{j}=\frac{1}{n+1}\mathrm{div}_g (X)(T_{k}(A_g))_{i}^{i}=(n-k)\sigma_k(A_g)=0.
\end{equation}
Since $\sigma_{k+1}(A_g) \leq 0$, by Lemma \ref{lem:Inequality} and \eqref{key term:vanishing} we obtain
\begin{equation}\label{eq:final 2}
\begin{aligned}
&-(k+1)\int_{\mathbb{S}_{+}^{n+1}}\sigma_{k+1}\left(A_{g}\right) u\ud v_{g}\\
\le & -\int_{\mathbb{S}^{n}}\psi_{f}uT_{k}(A_g)(\nu_{g},\nu_{g})\ud\sigma_{g} \\
 = & \int_{\mathbb{S}^{n}}T_{k}(A_g)(\nu_{g},\psi_{f}X)\ud\sigma_{g} \\
 = & \int_{\mathbb{S}_{+}^{n+1}}\mathrm{div}(T_{k}(A_g)\psi_{f}X)\ud v_{g}\\
 = & \int_{\mathbb{S}_{+}^{n+1}}\big[T_{k}(A_g)_{j}^{i}(\psi_{f})_{i}X^{j}+T_{k}(A_g)_{j}^{i}\psi_{f}\nabla_{i}X^{j}\big]\ud v_{g}\\
 = & \int_{\mathbb{S}_{+}^{n+1}}T_{k}(A_g)_{j}^{i}(\psi_{f})_{i}X^{j}\ud v_{g}.
\end{aligned}
\end{equation}

It remains to estimate
\begin{align*}
&T_{k}(A_g)_{j}^{i}(\psi_{f})_{i}X^{j} \\
= & T_{k}(A_g)_{j}^{i}(-A_{i}^{k}f_{k})X^{j} \\
 = & (\sigma_{k}(A_g)\delta_{j}^{i}-T_{k-1}(A_g)_{p}^{i}A_{j}^{p})(-A_{i}^{k}f_{k})X^{j} \\
 = & T_{k-1}(A_g)_{p}^{i}A_{j}^{p}X^{j}A_{i}^{k}f_{k} \\
 = & T_{k-1}(A_{g})(A_g X,A_g \nabla f) \\
 \le& \frac{1}{2}\big[u^{-1}T_{k-1}(A_{g})(A_g X,A_g X)+uT_{k-1}(A_{g})(A_g\nabla f,A_g\nabla f)\big].
\end{align*}

Standard computations give
\begin{align*}\label{eq:special relationship for sigmak}
T_{k+1}(A_g)_{p}^{q} = & \sigma_{k+1}(A_g)\delta_{p}^{q}-T_{k}(A_g)_{m}^{q}A_{p}^{m} \\
 = & \sigma_{k+1}(A_g)\delta_{p}^{q}-(\sigma_{k}(A_g)\delta_{m}^{q}-T_{k-1}(A)_{s}^{q}A_{m}^{s})A_{p}^{m} \\
 = & \sigma_{k+1}(A_g)\delta_{p}^{q}+T_{k-1}(A)_{s}^{q}A_{m}^{s}A_{p}^{m}
\end{align*}
and
\begin{align*}
T_{k-1}(A_{g})(A_g X,A_g X) = & T_{k-1}(A_g)_{j}^{i}A_{p}^{j}X^{p}A_{i}^{q}X_{q}=X_{q}A_{i}^{q}T_{k-1}(A_g)_{j}^{i}A_{p}^{j}X^{p}.
\end{align*}
Notice that $T_{k-1}(A_g)$ is semi-positive definite, so is $A_g T_{k-1}(A_g)A_g$. Then we obtain an upper bound estimate for eigenvalues of $A_{i}^{q}T_{k-1}(A_g)_{j}^{i}A_{p}^{j}$:
\begin{align*}
\lambda(A_{i}^{q}T_{k-1}(A_g)_{j}^{i}A_{p}^{j})\le& \mathrm{tr}(A_{i}^{q} T_{k-1}(A_g)_{j}^{i}A_{p}^{j})\\
=&A_{i}^{p}(T_{k-1}A_g)_{j}^{i}A_{p}^{j}\\
=&\mathrm{tr} T_{k+1}(A_g)-(n+1)\sigma_{k+1}(A_g)=-(k+1)\sigma_{k+1}(A_{g}).
\end{align*}
Consequently, putting these facts together yields 
\begin{equation}
\big|T_{k}(A_g)_{j}^{i}(\psi_{f})_{i}X^{j}\big|\le-\frac{k+1}{2}\big(u^{-1}\sigma_{k+1}(A_{g})|X|_{g}^{2}+u\sigma_{k+1}(A_{g})|\nabla f|_{g}^{2}\big).\label{eq:cauchy 1}
\end{equation}

Notice that 
\[
|\nabla f|_{g}^{2}=1+2\psi_{f}f
\]
 by Proposition \ref{thm:various properties of uxn}, and
\[
|X|_{g}^{2}=u^{2}|\nabla x_{n+2}|_{g_{0}}^2=u^{2}(1-x_{n+2}^{2}).
\]
Then we have
\begin{equation}\label{eq:cauchy 2}
\begin{aligned}
 & \frac{1}{2}\left[u^{-1}(-\sigma_{k+1}(A_{g}))|X|_{g}^{2}+u(-\sigma_{k+1}(A_{g}))|\nabla f|_{g}^{2}\right]\\
\le & -u\sigma_{k+1}(A_{g})-u\sigma_{k+1}(A_{g})\psi_{f}f+\frac{1}{2}u\sigma_{k+1}(A_{g})x_{n+2}^{2}.
\end{aligned}
\end{equation}

Therefore, by (\ref{eq:final 2})-(\ref{eq:cauchy 2}) we obtain
\begin{align*}
&-\int_{\mathbb{S}_{+}^{n+1}}\sigma_{k+1}\left(A_{g}\right)u \ud v_{g}\\
\le&\int_{\mathbb{S}_{+}^{n+1}}\big[-u\sigma_{k+1}(A_{g})-u\sigma_{k+1}(A_{g})\psi_{f}f+\frac{1}{2}u\sigma_{k+1}(A_{g})x_{n+2}^{2}\big]\ud v_{g}.
\end{align*}
This together with (\ref{eq:nonnegativity of psif}) yields 
\[
0\le\int_{\mathbb{S}_{+}^{n+1}}\big[-u\sigma_{k+1}(A_{g})\psi_{f}f+\frac{1}{2}u\sigma_{k+1}(A_{g})x_{n+2}^{2}\big]\ud v_{g}\le0.
\]
Thus, we obtain $\sigma_{k+1}(A_{g})=0$ in $\mathbb{S}_{+}^{n+1}$,
and the inequality in \eqref{eq:final 2} becomes an equality. Then we conclude  from the equality case in Lemma
\ref{lem:Inequality} that $A_{g}^{\top}=0$ on $\mathbb{S}^{n}$. 

Using the boundary condition $c_0=\mathcal{B}_{k}^{g}=\frac{n!}{(n+1-k)!(2k-1)!!}H_{g}^{2k-1}$ to obtain  $H_{g}=(\frac{c_{0}(n+1-k)!(2k-1)!!}{n!})^{\frac{1}{2k-1}}:=H_0$
on $\mathbb{S}^{n}$. By Codazzi-Gauss equations, we obtain $A_{\nu\alpha}=(H_{g})_{\alpha}=0$
and
\[
0=A_{\alpha\beta}=\overline A_{\alpha\beta}-\frac{H_{g}^{2}}{2}\bar g_{\alpha\beta}\qquad\mathrm{on~~} \mathbb{S}^{n}.
\]
Then by Riemann curvature tensor decomposition we know
\[
\bar{R}_{\alpha\beta\gamma\eta}=H_{g}^{2}(\overline{g}_{\alpha\gamma}\overline{g}_{\beta\eta}-\overline{g}_{\alpha\eta}\overline{g}_{\beta\gamma}) \qquad \mathrm{on~~} \mathbb{S}^{n}.
\]
that is,  $\bar g=g\big|_{\pa \mathbb{S}_+^{n+1}}$ is the round metric with constant sectional curvature $H_0^{2}$. Also, $\sigma_{1}(A_{g})\ge0$ in $\mathbb{S}_{+}^{n+1}$ by $g \in \overline{\mathcal C_k^+}$.

Let $\pi: x \in \mathbb{S}^{N}\setminus \{\mathsf{N}\}\mapsto z  \in  \mathbb{R}^{N}$ be the stereographic projection from the north pole $\mathsf N$. Consider $(\pi^{-1})^*(H_0^2 g)=v(z)^{\frac{4}{N-2}}|\mathrm{d}z|^2$ in the unit ball $\mathbb B_1 \subset \R^{N}$. Then $v$ satisfies 
$$-\Delta v\ge 0 \quad \mathrm{in~~}\mathbb B_1; \quad  v=1,\quad \partial_\nu v=0  \quad \mathrm{on~~}  \partial \mathbb B_1.$$
Consequently, we conclude that $v$ is harmonic and from  the strong maximum principle that $v=1$ in $\overline{\mathbb B_1}$.

Therefore, we conclude that $g$ is a flat metric in $\mathbb{S}_{+}^{n+1}$.
\end{proof}

\noindent{\em On the use of AI.}~~  Applying the Poincar\'e-Einstein hemisphere model to find the function $f=ux_{n+2}$ in  Proposition \ref{thm:various properties of uxn}, as well as Gover-Nurowski's paper \cite{GoverNurowski2006}, was suggested by ChatGPT 5.6 Pro.
The authors  verified all arguments and take full responsibility for the completeness and correctness of the proofs.

\bibliography{references}
\bibliographystyle{amsplain}
\end{document}